\documentclass[a4paper,10pt]{article}
\usepackage{amsmath,amsthm,xypic,
amssymb,amsfonts,
pdfsync,stmaryrd,mathrsfs,yfonts}
\usepackage[capitalise]{cleveref}
\newtheorem{thm}{Theorem}[section]

\newtheorem{prop}[thm]{Proposition}
\newtheorem{cor}[thm]{Corollary}

\newtheorem{lem}[thm]{Lemma}
\theoremstyle{definition}
\newtheorem{defi}[thm]{Definition}

\newtheorem{rem}[thm]{Remark}

\newtheorem{notation}[thm]{Notation}
\newtheorem{exm}[thm]{Example}

\theoremstyle{plain}

\newcommand{\intt}{\mathop{\mathrm{Int}}}

\newcommand{\diag}{\mathop{\mathrm{diag}}}
\newcommand{\lla}{\langle\!\langle}
\newcommand{\rra}{\rangle\!\rangle}

\renewcommand{\phi}{\varphi}

\newcommand{\alt}{\mathop{\mathrm{Alt}}}

\newcommand{\sym}{\mathop{\mathrm{Sym}}}

\newcommand{\End}{\mathop{\mathrm{End}}}
\newcommand{\Hom}{\mathop{\mathrm{Hom}}}

\newcommand{\disc}{\mathop{\mathrm{disc}}}

\newcommand\iso{\xrightarrow{
   \,\smash{\raisebox{-0.3ex}{\ensuremath{\scriptstyle\sim}}}\,}}

\author{A.-H. Nokhodkar}
\date{}
\title
{
Orthogonal involutions and totally singular quadratic forms
in characteristic two
}

\begin{document}
\maketitle

\begin{abstract}
We associate to every central simple algebra with involution of ortho\-gonal type in characteristic two a totally singular quadratic form which reflects certain anisotropy properties of the involution.
It is shown that this quadratic form can be used to classify totally decomposable algebras with orthogonal involution.
Also, using this form, a criterion is obtained for an orthogonal involution on a split algebra to be conjugated to the transpose involution.
\end{abstract}

\section{Introduction}
The theory of algebras with involution is closely related to the theory of bilinear forms.
Assigning to every symmetric or anti-symmetric regular bilinear form on a finite-dimensional vector space $V$ its adjoint involution induces a one-to-one correspondence between the similarity classes of bilinear forms on $V$ and involutions of the first kind on the endomorphism algebra $\End_F(V)$ (see \cite[p. 1]{knus}).
Under this correspondence several basic properties of involutions reflect analogous properties of bilinear forms.
For example a symmetric or anti-symmetric bilinear form is isotropic (resp. anisotropic, hyperbolic) if and only if its adjoint involution is isotropic (resp. anisotropic, hyperbolic).

Let $(A,\sigma)$ be a totally decomposable algebra with orthogonal involution over a field $F$ of characteristic two.
In \cite{dolphin3}, a bilinear Pfister form $\mathfrak{Pf}(A,\sigma)$ was associated to $(A,\sigma)$, which determines the isotropy behaviour of $(A,\sigma)$.
It was shown that for every splitting field $K$ of $A$, the involution $\sigma_K$ on $A_K$ becomes adjoint to the bilinear form $\mathfrak{Pf}(A,\sigma)_K$ (see \cite[(7.5)]{dolphin3}).
Also, according to \cite[(6.5)]{mn1} this invariant can be used to determine the conjugacy class of the involution $\sigma$ on $A$.

In this work we associate to every algebra with orthogonal involution $(A,\sigma)$ in characteristic two a pair $(S(A,\sigma),q_\sigma)$ where $S(A,\sigma)$ is a subalgebra of $A$ and $q_\sigma$ is a totally singular quadratic form on $S(A,\sigma)$.
Using this pair, we find in \cref{direct} some criteria for $\sigma$ to be {\it direct} (a stronger notion than anisotropy defined in \cite{dolphin2}).
As a consequence, several sufficient conditions for anisotropy of $\sigma$ are obtained in \cref{cor}.
Further, it is shown in \cref{tran} that $q_\sigma$ can classify the transpose involution on a split algebra.
For the case where $(A,\sigma)$ is totally decomposable, using the quadratic space $(S(A,\sigma),q_\sigma)$ we will complement some results of \cite{dolphin3} and \cite{mn1} in \cref{isotropic} and state several necessary and sufficient conditions for $\sigma$ to be anisotropic.
Finally, we shall see in \cref{main1} and \cref{main2} that the form $q_\sigma$ can be used to find a classification of totally decomposable algebras with orthogonal involution in characteristic two.

\section{Preliminaries}
In this paper, $F$ denote a field of characteristic two.

Let $V$ be a vector space of finite dimension over $F$.
A bilinear form $\mathfrak{b}:V\times V\rightarrow F$ is called {\it isotropic} if $\mathfrak{b}(v,v)=0$ for some $0\neq v\in V$.
Otherwise, $\mathfrak{b}$ is called {\it anisotropic}.
We say that $\mathfrak{b}$ {\it represents} $\alpha\in F$ if $\mathfrak{b}(v,v)=\alpha$ for some nonzero vector $v\in V$.
The set of all elements represented by $\mathfrak{b}$ is denoted by $D(\mathfrak{b})$.
If $K/F$ is a field extension, we denote by $\mathfrak{b}_K$ the {\it scalar extension} of $\mathfrak{b}$ to $K$.
For $\alpha_1,\cdots,\alpha_n\in F$ the diagonal bilinear form $\sum_{i=1}^n\alpha_ix_iy_i$  is denoted by $\langle\alpha_1,\cdots,\alpha_n\rangle$.
Also, the form $\bigotimes_{i=1}^n\langle1,\alpha_i\rangle$ is called a {\it bilinear Pfister form} and is denoted by $\lla\alpha_1,\cdots,\alpha_n\rra$.
If $\mathfrak{b}$ is a bilinear Pfister form over $F$, there exists a bilinear form $\mathfrak{b}'$, called the {\it pure subform} of $\mathfrak{b}$, such that $\mathfrak{b}\simeq\langle1\rangle\perp\mathfrak{b}'$.
As observed in \cite[p. 16]{arason} the pure subform of $\mathfrak{b}$ is uniquely determined, up to isometry.
According to \cite[(6.5)]{elman}, the form $\mathfrak{b}$ is anisotropic if and only if $\dim_{F^2}Q(\mathfrak{b})=2^n$, where $Q(\mathfrak{b})=D(\mathfrak{b})\cup\{0\}$.

A {\it quadratic form} on $V$ is a map $q:V\rightarrow F$ such that $(i)\ q(\alpha v)=\alpha^2 q(v)$ for $\alpha\in F$ and $v\in V$; $(ii)$ the map $\mathfrak{b}_q:V\times V\rightarrow F$ given by $\mathfrak{b}(v,w)=q(v+w)-q(v)-q(w)$ is an $F$-bilinear form.
A quadratic form $q$ is called {\it isotropic} if $q(v)=0$ for some $0\neq v\in V$ and {\it anisotropic} otherwise.
For a quadratic space $(V,q)$ over $F$ we use the notation $D(q)=\{q(v)\mid0\neq v\in V\}$ and $Q(q)=D(q)\cup\{0\}$.
The scalar extension of $q$ to an extension $K$ of $F$ is denoted by $q_K$.
A quadratic form $q$ is called {\it totally singular} if $\mathfrak{b}_q$ is the zero map.
For $\alpha_1,\cdots,\alpha_n\in F$ the totally singular quadratic form $\sum_{i=1}^n\alpha_ix_i^2$ is denoted by $\langle\alpha_1,\cdots,\alpha_n\rangle_q$.

Let $A$ be a central simple algebra over $F$ and let $\sigma$ be an {\it involution} on $A$, i.e., an anti-automorphism of $A$ of order two.
We say that $\sigma$ is of {\it the first kind} if it restricts to the identity map on $F$.
For a symmetric bilinear space $(V,\mathfrak{b})$ over $F$ we denote by $\sigma_\mathfrak{b}$ the adjoint involution of $\End_F(V)$ with respect to $\mathfrak{b}$ (see \cite[p. 1]{knus}).
An involution $\sigma$ of the first kind on $A$ is called {\it symplectic} if it becomes adjoint to an {\it alternating} bilinear form over a splitting field of $A$ (i.e., a bilinear form $\mathfrak{b}$ with $D(\mathfrak{b})=0$).
Otherwise, $\sigma$ is called {\it orthogonal}.
According to \cite[(2.6)]{knus}, $\sigma$ is orthogonal if and only if $1\notin\alt(A,\sigma)$, where
$\alt(A,\sigma)=\{x-\sigma(x)\mid x\in A\}$.
The discriminant of an orthogonal involution $\sigma$ is denoted by $\disc\sigma$ (see \cite[(7.1)]{knus}).
An involution $\sigma$ is called {\it isotropic} if $\sigma(x)x=0$ for some nonzero element $x\in A$.
Otherwise, $\sigma$ is called {\it anisotropic}.

We will frequently use the following result.
Recall that if $u$ is a unit in a central simple algebra $A$, the {\it inner automorphism} of $A$ induced by $u$ is defined as $\intt(u)(x)=uxu^{-1}$ for $x\in A$.
\begin{prop}\label{21}
Let $\alpha_1,\cdots,\alpha_n\in F^\times$ and let $u\in M_n(F)$ be the diagonal matrix $\diag(\alpha_1,\cdots,\alpha_n)$.
Consider the involution $\sigma=\intt(u)\circ t$ on $M_n(F)$, where $t$ is the transpose involution.
If $(V,\mathfrak{b})$ is the diagonal bilinear space $\langle\alpha_1,\cdots,\alpha_n\rangle$, then $(M_n(F),\sigma)\simeq(\End_F(V),\sigma_\mathfrak{b})$.
\end{prop}

\begin{proof}
See \cite[pp. 13-14]{knus}.
\end{proof}

\section{The alternator form}
For a central simple algebra with orthogonal involution $(A,\sigma)$ over $F$ we use the following notation:
\[S(A,\sigma)=\{x\in A\mid \sigma(x)x\in F\oplus\alt(A,\sigma)\}.\]
In other words, $x\in S(A,\sigma)$ if and only if there exists a unique element $\alpha\in F$ such that $\sigma(x)x+\alpha\in\alt(A,\sigma)$.
We denote the element $\alpha$ by $q_\sigma(x)$.
Hence, $q_\sigma:S(A,\sigma)\rightarrow F$ is a map satisfying
\[\sigma(x)x+q_\sigma(x)\in\alt(A,\sigma)\quad {\rm for} \ x\in S(A,\sigma).\]
\begin{lem}\label{alt}
Let $(A,\sigma)$ be a central simple algebra with involution over $F$.
If $x\in\alt(A,\sigma)$, then $\sigma(y)xy\in\alt(A,\sigma)$ for every $y\in A$.
\end{lem}
\begin{proof}
Write $x=z-\sigma(z)$ for some $z\in A$.
Then
\[\sigma(y)xy=\sigma(y)(z-\sigma(z))y=\sigma(y)zy-\sigma(\sigma(y)zy)\in\alt(A,\sigma).\qedhere\]
\end{proof}

\begin{lem}\label{algebra}
Let $(A,\sigma)$ be a central simple algebra with orthogonal involution over $F$.
Then
\begin{itemize}
  \item[$(i)$] $S(A,\sigma)$ is a (unitary) $F$-subalgebra of $A$.
  \item[$(ii)$] $q_\sigma(\lambda x+y)=\lambda^2q_\sigma(x)+q_\sigma(y)$ and $q_\sigma(xy)=q_\sigma(x)q_\sigma(y)$ for every $\lambda\in F$ and $x,y\in S(A,\sigma)$.
\end{itemize}
\end{lem}
\begin{proof}
For every $\lambda\in F$ we have $\sigma(\lambda)\lambda+\lambda^2=0\in\alt(A,\sigma)$, so $F\subseteq S(A,\sigma)$.
Let $x,y\in S(A,\sigma)$ and $\lambda\in F$.
Set $\alpha=q_\sigma(x)\in F$ and $\beta=q_\sigma(y)\in F$.
Then
\[\sigma(\lambda x+y)(\lambda x+y)+\lambda^2\alpha+\beta=\lambda^2(\sigma(x)x+\alpha)+(\sigma(y)y+\beta)+\sigma(\lambda x)y-\sigma(\sigma(\lambda x)y).\]
Hence, $\sigma(\lambda x+y)(\lambda x+y)+\lambda^2\alpha+\beta\in\alt(A,\sigma)$, which implies that $\lambda x+y\in S(A,\sigma)$ and $q_\sigma(\lambda x+y)=\lambda^2\alpha+\beta=\lambda^2q_\sigma(x)+q_\sigma(y)$.
Similarly, using \cref{alt} we have
\begin{align*}
  \sigma(xy)(xy)+\alpha\beta&=\sigma(y)\sigma(x)xy+\alpha\beta=\sigma(y)(\sigma(x)x+\alpha)y+\alpha\sigma(y)y+\alpha\beta\\
  &=\sigma(y)(\sigma(x)x+\alpha)y+\alpha(\sigma(y)y+\beta)\in\alt(A,\sigma).
\end{align*}
Hence, $xy\in S(A,\sigma)$ and $q_\sigma(xy)=\alpha\beta=q_\sigma(x)q_\sigma(y)$, proving the result.
\end{proof}

\begin{cor}
Let $(A,\sigma)$ be a central simple algebra with orthogonal involution over $F$.
Then $q_\sigma$ is a totally singular quadratic form on $S(A,\sigma)$.
\end{cor}

\begin{proof}
The result follows from \cref{algebra} $(ii)$.
\end{proof}

\begin{defi}
Let $(A,\sigma)$ be a central simple algebras with ortho\-gonal involution over $F$.
We call $S(A,\sigma)$ the {\it alternator subalgebra} of $(A,\sigma)$.
We also call the quadratic form $q_\sigma$ the {\it alternator form} of $(A,\sigma)$.
\end{defi}

\begin{lem}\label{iso}
Let $(A,\sigma)$ and $(A',\sigma')$ be two central simple algebras with ortho\-gonal involution over $F$.
If $f:(A,\sigma)\iso(A',\sigma')$ is an isomorphism of algebras with involution, the restriction of $f$ to $S(A,\sigma)$ defines an isometry $(S(A,\sigma),q_\sigma)\penalty 0\iso (S(A',\sigma'),q_{\sigma'})$.
\end{lem}

\begin{proof}
For every $x\in S(A,\sigma)$ we have $\sigma(x)x+q_\sigma(x)\in\alt(A,\sigma)$, which implies that $f(\sigma(x)x+q_\sigma(x))\in\alt(A',\sigma')$.
Hence, $\sigma'(f(x))f(x)+q_\sigma(x)\in\alt(A',\sigma')$, i.e., $f(x)\in S(A',\sigma')$ and $q_{\sigma'}(f(x))=q_\sigma(x)$.
\end{proof}

The following definition was given in \cite{dolphin2}.
\begin{defi}
An involution $\sigma$ on a central simple algebra $A$ is called {\it direct} if for every $x\in A$ the condition $\sigma(x)x\in\alt(A,\sigma)$ implies that $x=0$.
\end{defi}
\begin{thm}\label{direct}
For an orthogonal involution $\sigma$ on a central simple $F$-algebra $A$ the following conditions are equivalent:
\begin{itemize}
  \item[$(1)$] $\sigma$ is direct.
  \item[$(2)$] $q_\sigma$ is anisotropic.
  \item[$(3)$] $S(A,\sigma)$ is a field.
  \end{itemize}
Moreover, if these conditions hold, then $x^2\in F$ for all $x\in S(A,\sigma)$.
\end{thm}

\begin{proof}
The implication $(1)\Rightarrow (2)$ is evident.

$(2)\Rightarrow(3):$
Since $S(A,\sigma)$ is a subalgebra of $A$, it suffices to show that (i) $S(A,\sigma)$ contains no zero devisor;
(ii) $x^{-1}\in S(A,\sigma)$ for every nonzero element $x\in S(A,\sigma)$; and
(iii) $S(A,\sigma)$ is commutative.

Suppose that $xy=0$ for some $0\neq x\in S(A,\sigma)$ and $y\in A$.
Set $\alpha=q_\sigma(x)$, so that $\sigma(x)x+\alpha\in\alt(A,\sigma)$.
By \cref{alt} we have
\begin{equation*}
\alpha\sigma(y)y=0+\alpha\sigma(y)y=\sigma(y)\sigma(x)xy+\alpha\sigma(y)y=\sigma(y)(\sigma(x)x+\alpha)y\in\alt(A,\sigma).
\end{equation*}
Since $q_\sigma$ is anisotropic we have $\alpha\neq0$.
Hence, $y\in S(A,\sigma)$ and $q_\sigma(y)=0$.
Again, the anisotropy of $q_\sigma$ implies that $y=0$, proving (i).

To prove (ii) let $0\neq x\in S(A,\sigma)$ and $\alpha=q_\sigma(x)\neq 0$, so that $\sigma(x)x+\alpha\in\alt(A,\sigma)$.
By \cref{alt} we have $\sigma(x^{-1})(\sigma(x)x+\alpha)x^{-1}\in\alt(A,\sigma)$, i.e., $\alpha\sigma(x^{-1})(x^{-1})+1\in\alt(A,\sigma)$.
It follows that $\sigma(x^{-1})(x^{-1})+\alpha^{-1}\in\alt(A,\sigma)$, so $x^{-1}\in S(A,\sigma)$.

Finally, if $x,y\in S(A,\sigma)$, then \cref{algebra} implies that
\[q_\sigma(xy+yx)=q_\sigma(xy)+q_\sigma(yx)=q_\sigma(x)q_\sigma(y)+q_\sigma(y)q_\sigma(x)=0.\]
Since $q_\sigma$ is anisotropic we get $xy=yx$, proving (iii).

$(3)\Rightarrow(1):$ Suppose that $\sigma(x)x\in\alt(A,\sigma)$ for some $x\in A$.
Then $x\in S(A,\sigma)$ and $q_\sigma(x)=0$.
If $x\neq0$ then $x$ is a unit, because $S(A,\sigma)$ is a field.
Since $\sigma(x)x\in\alt(A,\sigma)$, \cref{alt} implies that $1=\sigma(x^{-1})\sigma(x)xx^{-1}\in\alt(A,\sigma)$.
This contradicts the orthogonality of $\sigma$.

To prove the last statement of the result, let $x\in S(A,\sigma)$ and set $\alpha=q_\sigma(x)$.
By \Cref{algebra} we have $q_\sigma(x^2)=\alpha^2=q_\sigma(\alpha)$, so $q_\sigma(x^2+\alpha)=0$.
As $q_\sigma$ is anisotropic, we get $x^2=\alpha\in F$.
\end{proof}

\begin{lem}{\rm(\cite[(2.26)]{knus})}\label{gen}
Let $(A,\sigma)$ be a central simple algebra with orthogonal involution over $F$.
Then the set $\sym(A,\sigma)$ generates $A$ as an (associative) $F$-algebra.
\end{lem}
\begin{proof}
If $\deg_FA>2$, the result follows from \cite[(2.26)]{knus}.
Suppose that $\deg_FA=2$.
Let $B\subseteq A$ be the subalgebra of $A$ generated by $\sym(A,\sigma)$.
Using the idea of the proof of \cite[(2.26)]{knus}, it is enough to show that $B_K=A_K$ for some extension $K$ of $F$.
Choose an extension $K/F$ with $(A,\sigma)_K\simeq(M_2(K),t)$.
The conclusion now easily follows by identifying $\sym(A,\sigma)_K$ with the set of matrices of the form $\left(\begin{smallmatrix}
                                                                           a & b \\
                                                                           b & c
                                                                         \end{smallmatrix}\right)$, where $a,b,c\in K$.
\end{proof}

\begin{prop}\label{sym}
Let $(A,\sigma)$ be a central simple algebra with orthogonal involution over $F$.
If $S(A,\sigma)\subseteq \sym(A,\sigma)$, then $\sigma$ is direct.
\end{prop}

\begin{proof}
Suppose that $\sigma(x)x\in\alt(A,\sigma)$ for some $x\in A$.
Then $x\in S(A,\sigma)$, which implies that $\sigma(x)=x$.
We claim that $\sym(A,\sigma)\subseteq C_A(x)$, where $C_A(x)$ is the centralizer of $x$ in $A$.
Let $y\in\sym(A,\sigma)$.
By \cref{alt} we have
\[\sigma(xy)xy=\sigma(y)\sigma(x)xy\in\alt(A,\sigma),\]
hence $xy\in S(A,\sigma)\subseteq \sym(A,\sigma)$.
It follows that $yx=\sigma(y)\sigma(x)=\sigma(xy)=xy$, because the elements $x$, $y$ and $xy$ are all symmetric.
This proves the claim.
By \cref{gen}, $\sym(A,\sigma)$ generates $A$ as an $F$-algebra.
Hence, $C_A(x)=A$, i.e., $x\in F$.
Since $\sigma$ is orthogonal, the condition $\sigma(x)x\in\alt(A,\sigma)$ implies that $x=0$, so $\sigma$ is direct.
\end{proof}

Since every direct involution is anisotropic, one can use \cref{direct} and \cref{sym} to find some sufficient conditions for anisotropy of orthogonal involutions:

\begin{cor}\label{cor}
Let $(A,\sigma)$ be a central simple algebra with orthogonal involution over $F$.
If any of these conditions is satisfied, then $\sigma$ is anisotropic:
$(i)$ $q_\sigma$ is anisotropic. $(ii)$ $S(A,\sigma)$ is a field. $(iii)$ $S(A,\sigma)\subseteq \sym(A,\sigma)$.
\end{cor}

\begin{lem}\label{split}
Let $(V,\mathfrak{b})$ be a symmetric non-alternating bilinear space over $F$.
\begin{itemize}
  \item [$(i)$] If $x\in\End_F(V)$, then $\mathfrak{b}(x(v),x(v))=q_{\sigma_\mathfrak{b}}(x)\mathfrak{b}(v,v)$ for every $v\in V$.
  \item [$(ii)$] If $\mathfrak{b}$ represents $1$, then $D(q_{\sigma_\mathfrak{b}})\subseteq D(\mathfrak{b})$.
\end{itemize}
\end{lem}

\begin{proof}
Let $\alpha=q_{\sigma_\mathfrak{b}}(x)$, so that $\sigma_\mathfrak{b}(x)x+\alpha\in\alt(\End_F(V),\sigma_\mathfrak{b})$.
Write $\sigma_\mathfrak{b}(x)x=y-\sigma_\mathfrak{b}(y)-\alpha$ for some $y\in\End_F(V)$.
For every $v\in V$  we have
\begin{align*}
\mathfrak{b}(x(v),x(v))&=\mathfrak{b}((\sigma_\mathfrak{b}(x)x)(v),v)=\mathfrak{b}((y-\sigma_\mathfrak{b}(y)-\alpha)(v),v)\\
&=\mathfrak{b}(y(v),v)-\mathfrak{b}(\sigma_\mathfrak{b}(y)(v),v)-\alpha\mathfrak{b}(v,v)\\
&=\mathfrak{b}(y(v),v)-\mathfrak{b}(v,y(v))-\alpha\mathfrak{b}(v,v)=\alpha\mathfrak{b}(v,v).
\end{align*}
This proves the first part.
The second part follows by applying $(i)$ to a vector $v\in V$ with $\mathfrak{b}(v,v)=1$.
\end{proof}

\begin{rem}
The converse of \cref{sym} does not hold in general.
To construct a counter-example, let $\lla\alpha,\beta\rra$ be an anisotropic bilinear Pfister form over $F$  and let $u$ be the diagonal matrix $\diag(1,\alpha,\beta,\alpha\beta+1)$.
Consider the involution $\sigma=\intt(u)\circ t$ on $M_4(F)$.
By \cref{21} we have $(M_4(F),\sigma)\simeq(\End_F(V),\sigma_\mathfrak{b})$, where $(V,\mathfrak{b})$ is the diagonal bilinear space $\langle1,\alpha,\beta,\alpha\beta+1\rangle$.
Since $\mathfrak{b}$ is anisotropic, the form $q_\sigma$ is also anisotropic by \cref{split} $(ii)$.
Hence, $\sigma$ is direct by \cref{direct}.
Let
\[x=\left(\begin{smallmatrix}
                0 & (\alpha\beta)^{-1} & \beta^{-1} & 0\\
                0 & 0 & 0 & \alpha(1+\alpha\beta)^{-1} \\
                1 & 0 & 0 & (1+\alpha\beta)^{-1} \\
                0 & 1+(\alpha\beta)^{-1} & 0 & 0
              \end{smallmatrix}\right)\in M_4(F).\]
Computations shows that $(\sigma(x)x+\beta^{-1}I_4)u\in\alt(M_4(F),t)$, where $I_4$ is the $4\times4$ identity matrix.
By \cite[(2.7)]{knus} we have  $\sigma(x)x+\beta^{-1}I_4\in\alt(M_4(F),\sigma)$, so $x\in S(M_4(F),\sigma)$.
On the other hand
\[xu=\left(\begin{smallmatrix}
               0 & \beta^{-1} & 1 & 0 \\
               0 & 0 & 0 & \alpha \\
               1 & 0 &  0 & 1 \\
               0 & \alpha+\beta^{-1} & 0 & 0
             \end{smallmatrix}\right)\notin\sym(M_4(F),t),\]
which implies that $x\notin\sym(M_4(F),\sigma)$, thanks to \cite[(2.7)]{knus}.
It follows that $x\in S(M_4(F),\sigma)\setminus \sym(M_4(F),\sigma)$.
\end{rem}

\begin{lem}\label{f2}
Let $t$ be the transpose involution on $M_n(F)$.
Then
\begin{itemize}
  \item[$(i)$] $Q(q_t)=F^2$.
  \item[$(ii)$] $q_t\simeq\langle1\rangle_q\perp(n^2-n)\times\langle0\rangle_q$.
\end{itemize}
\end{lem}
\begin{proof}
$(i)$ Clearly, we have $F^2\subseteq Q(q_t)$.
The converse inclusion follows from \cref{split} $(ii)$ and the fact that the transpose involution is the adjoint involution of the bilinear form $n\times\langle1\rangle$ (see \cref{21}).

$(ii)$ Using the first part and the isometry $\langle1,1\rangle_q\simeq\langle1,0\rangle_q$ we have $q_t\simeq\langle1\rangle_q\perp k\times\langle0\rangle_q$ for some non-negative integer $k$.
We claim that $k=n^2-n$.
Let $W=\{x\in S(M_n(F),t)\mid q_t(x)=0\}$.
Then $W$ is a $k$-dimensional vector space over $F$.
A matrix $x=(x_{ij})\in M_n(F)$ belongs to $W$ if and only if $x^tx\in\alt(M_n(F),t)$.
Note that $\alt(M_n(F),t)$ is the set of symmetric matrices with zero diagonal.
Hence, $x^tx\in\alt(M_n(F),t)$ if and only if
$\sum_{i=1}^nx_{ij}^2=0$ for $j=1,\cdots,n$, or equivalently
\begin{equation}\label{eq3}
\textstyle\sum_{i=1}^nx_{ij}=0,\quad {\rm for} \ j=1,\cdots,n.
\end{equation}
Hence, $W$ is the set of answers of the homogeneous system of linear equations (\ref{eq3}) with $n^2$ unknowns $x_{ij}$, $i,j=1,\cdots,n$.
This set is a vector space of dimension $n^2-n$ over $F$,  hence $k=n^2-n$.
\end{proof}

\begin{thm}\label{tran}
Let $(A,\sigma)$ be a central simple algebra of degree $n$ with ortho\-gonal involution over $F$.
If $A$ splits, then $(A,\sigma)\simeq(M_n(F),t)$ if and only if $q_\sigma\simeq\langle1\rangle_q\perp(n^2-n)\times\langle0\rangle_q$.
\end{thm}

\begin{proof}
The `only if' implication follows from \cref{iso} and \cref{f2}.

To prove the converse, we may identify $(A,\sigma)=(\End_F(V),\sigma_{\mathfrak{b}})$, where $(V,\mathfrak{b})$ is a symmetric non-alternating bilinear space over $F$.
By \cite[(2.1)]{lagh} there exist unique integers $m,k$ and $a_1,\cdots,a_m\in F$ such that
\begin{equation}\label{eq4}
\mathfrak{b}\simeq\mathfrak{b}_{an}\perp\mathbb{M}(a_1)\perp\cdots\perp\mathbb{M}(a_m)\perp k\times\mathbb{H},
\end{equation}
where $\mathfrak{b}_{an}$ is an anisotropic bilinear form, $\mathbb{M}(a_i)=\langle a_i,a_i\rangle$ and $\mathbb{H}$ is the hyperbolic plane.
Let
\[U=\{x\in S({\End}_F(V),\sigma_{\mathfrak{b}})\mid q_{\sigma_\mathfrak{b}}(x)=0\}\quad {\rm and} \quad W=\{v\in V\mid \mathfrak{b}(v,v)=0\}.\]
Then $U$ and $W$ are two vector spaces over $F$.
We have $\dim_F U=n^2-n$, because $q_\sigma\simeq\langle1\rangle_q\perp(n^2-n)\times\langle0\rangle_q$.
Also, by \cite[(2.1)]{lagh} we have $\dim_FW=2k+m$.
On the other hand \cref{split} $(i)$ implies that $x(V)\subseteq W$ for every $x\in U$, i.e., $U\subseteq \Hom_F(V,W)$.
By dimension count we get $\dim_FW\geqslant n-1$.
However, we have $W\neq V$, because $\mathfrak{b}$ is not alternating.
Hence, $\dim_FW=n-1$, i.e., $2k+m=n-1$.
By (\ref{eq4}) we have $n=\dim_F\mathfrak{b}_{an}+2k+2m$, which implies that $\dim_F\mathfrak{b}_{an}+m=1$.
It follows that either $\mathfrak{b}_{an}$ is trivial and $m=1$ or $\dim_F\mathfrak{b}_{an}=1$ and $m=0$.
Hence, either $\mathfrak{b}\simeq\langle\alpha,\alpha\rangle\perp k\times\mathbb{H}$ or $\mathfrak{b}\simeq\langle\alpha\rangle\perp k\times\mathbb{H}$  for some $\alpha\in F^\times$.
Using the isometry $\langle\alpha\rangle\perp\mathbb{H}\simeq\langle\alpha,\alpha,\alpha\rangle$ in \cite[(1.16)]{elman}, we get $\mathfrak{b}\simeq n\times\langle\alpha\rangle$.
Hence, $(A,\sigma)\simeq(M_n(F),t)$ by \cref{21}.
\end{proof}
The next example shows that for $n>2$ there exists a central simple algebra with orthogonal involution $(A,\sigma)$ of degree $n$ over $F$ such that $q_\sigma=\langle1\rangle_q$.
Note that for every algebra with involution $(A,\sigma)$ over $F$ we have $F\subseteq S(A,\sigma)$.
Hence, the form $\langle1\rangle_q$ is always a subform of the alternator form $q_\sigma$.

\begin{exm}\label{exm}
For $\alpha_1,\cdots,\alpha_n\in F^\times$ let $u$ be the diagonal $n\times n$ matrix $\diag(\alpha_1,\cdots,\alpha_n)$ and consider the involution $\sigma=\intt(u)\circ t$ on $M_n(F)$.
Let $\mathfrak{b}=\langle\alpha_1,\cdots,\alpha_n\rangle$.
By \cref{21} we have $(M_4(F),\sigma)\simeq(\End_F(V),\sigma_\mathfrak{b})$, where $V$ is an underlying vector space of $\mathfrak{b}$.
Let $x=(x_{ij})\in M_n(F)$.
We first claim that $x\in S(M_n(F),\sigma)$ with $q_\sigma(x)=\lambda\in F$ if and only if
\begin{align}\label{eq11}
\alpha_1^{-1}x_{1i}^2+\cdots+\alpha_n^{-1}x_{ni}^2=\alpha_i^{-1}\lambda \quad {\rm for}\ i=1,\cdots,n.
\end{align}
By definition $x\in S(M_n(F),\sigma)$ (with $q_\sigma(x)=\lambda$) if and only if
\[y:=\sigma(x)x+\lambda I_n\in\alt(M_n(F),\sigma).\]
By \cite[(2.4)]{knus} we have $\alt(M_n(F),\sigma)=u\cdot\alt(M_n(F),t)$, so $y\in \alt(M_n(F),\sigma)$ if and only if $u^{-1}y\in\alt(M_n(F),t)$.
Since $y\in\sym(M_n(F),\sigma)$ we have $u^{-1}y\in \sym(M_n(F),t)$ by \cite[(2.4)]{knus}.
Hence, $x\in S(M_n(F),\sigma)$ if and only if $u^{-1}y$ has zero diagonal.
Write $u^{-1}y=(y_{ij})$ for some $y_{ij}\in F$, $1\leqslant i,j\leqslant n$.
As $y=ux^tu^{-1}x+\lambda I_n$, computation shows that
\begin{align*}
y_{ii}=\alpha_1^{-1}x_{1i}^2+\cdots+\alpha_n^{-1}x_{ni}^2+\alpha_i^{-1}\lambda \quad {\rm for}\ i=1,\cdots,n.
\end{align*}
We have $y_{ii}=0$ if and only if (\ref{eq11}) is satisfied, as claimed.

Now, if $n>2$ and $\lla\alpha_1,\cdots,\alpha_n\rra$ is an anisotropic bilinear Pfister form, then (\ref{eq11}) implies that $\lambda=x_{11}^2=\cdots=x_{nn}^2$ and
$x_{ij}=0$ for $i\neq j$.
It follows that $\lambda\in F^2$ and $x$ is a scalar (matrix).
Hence, $S(A,\sigma)=F$ and $q_\sigma=\langle1\rangle_q$.
\end{exm}

\begin{rem}
\cref{exm} shows that the alternator form does not necessarily classify orthogonal involutions on a given central simple algebra.
Also, using \cref{exm} one can show that the inclusion $Q(q_{\sigma_{\mathfrak{b}}})\subseteq Q(\mathfrak{b})$ in \cref{split} could be strict.
Indeed, with the notation of \cref{exm} set $\mathfrak{b}'=\alpha_1\cdot\mathfrak{b}$.
Then $\mathfrak{b}'$ represents $1$ and $(\End_F(V),\sigma_\mathfrak{b})\simeq(\End_F(V),\sigma_{\mathfrak{b}'})$ (see \cite[p. 1]{knus}).
By \cref{iso} we have $Q(q_{\sigma_{\mathfrak{b}'}})=Q(q_{\sigma_\mathfrak{b}})=F^2$.
Since $\mathfrak{b}$ is anisotropic we have $\alpha_2\notin F^2$, hence $\alpha_2\in Q(\mathfrak{b}')\setminus  Q(q_{\sigma_{\mathfrak{b}'}})$.

\end{rem}

\section{Applications to totally decomposable involutions}
A {\it quaternion algebra} over $F$ is a central simple algebra of degree $2$.
An algebra with involution $(A,\sigma)$ over $F$ is called {\it totally decomposable} if it decomposes into tensor products of quaternion $F$-algebras with involution.
Let $(A,\sigma)\simeq\bigotimes_{i=1}^n(Q_i,\sigma_i)$ be a totally decomposable algebra with orthogonal involution over $F$.
By \cite[(2.23)]{knus} every $\sigma_i$ is an orthogonal involution.
Write $\disc\sigma_i=\alpha_iF^{\times2}\in F^\times/F^{\times2}$ for some $\alpha_i\in F^\times$, $i=1,\cdots,n$.
As in \cite{dolphin3} we denote the bilinear Pfister form $\lla\alpha_1,\cdots,\alpha_n\rra$ by $\mathfrak{Pf}(A,\sigma)$.
By \cite[(7.5)]{dolphin3}, $\mathfrak{Pf}(A,\sigma)$ is independent of the decomposition of $(A,\sigma)$.
Also, as observed in \cite{mn1} there exists a unique, up to isomorphism, $F$-algebra $\Phi(A,\sigma)\subseteq F\oplus\alt(A,\sigma)$ of dimension $2^n$ satisfying:
(i) $x^2\in F$ for $x\in\Phi(A,\sigma)$; (ii) the centralizer of $\Phi(A,\sigma)$ in $A$ coincides with $\Phi(A,\sigma)$ itself; (iii) $\Phi(A,\sigma)$ is generated, as an $F$-algebra by $n$ elements.
According to \cite[(5.5)]{mn1} the algebra $\Phi(A,\sigma)$ may be considered as an underlying vector space of $\mathfrak{Pf}(A,\sigma)$ such that
\[\mathfrak{Pf}(A,\sigma)(x,x)=x^2\quad {\rm for} \ x\in\Phi(A,\sigma).\]
Finally, let $x\in\Phi(A,\sigma)$ and set $\alpha=x^2\in F$.
Then $\sigma(x)x+\alpha=x^2+\alpha=0\in\alt(A,\sigma)$.
Hence, $\Phi(A,\sigma)\subseteq S(A,\sigma)$ and
\[q_\sigma(x)=\mathfrak{Pf}(A,\sigma)(x,x)=x^2\quad {\rm for}\ x\in\Phi(A,\sigma).\]

The following theorem complements some results in \cite{dolphin3} and \cite{mn1}.
\begin{thm}\label{isotropic}
For a totally decomposable algebra with orthogonal involution $(A,\sigma)$ over $F$ the following conditions are equivalent:
$(1)$ $\sigma$ is anisotropic.
$(2)$ $\sigma$ is direct.
$(3)$ $\mathfrak{Pf}(A,\sigma)$ is anisotropic.
$(4)$ $q_\sigma$ is anisotropic.
$(5)$ $\Phi(A,\sigma)$ is a field.
$(6)$ $S(A,\sigma)$ is a field.
$(7)$ $\Phi(A,\sigma)=S(A,\sigma)$.
$(8)$ $S(A,\sigma)\subseteq\sym(A,\sigma)$.
\\
In particular, if these conditions hold, the algebra $\Phi(A,\sigma)$ is uniquely determined.
\end{thm}

\begin{proof}
The equivalences $(1)\Leftrightarrow(2)\Leftrightarrow(3)$ can be found in \cite{dolphin3} (see \cite[(6.1), (6.2) and (7.5)]{dolphin3}).
The equivalences $(2)\Leftrightarrow(4)\Leftrightarrow(6)$ are proved in \cref{direct} and $(1)\Leftrightarrow(5)$ follows from \cite[(6.6)]{mn1} and \cite[(6.2)]{dolphin3}.
Suppose that is $S(A,\sigma)$ is a field.
Since $\Phi(A,\sigma)\subseteq S(A,\sigma)$ and $\Phi(A,\sigma)$ is maximal commutative, we get $\Phi(A,\sigma)=S(A,\sigma)$.
This proves $(6)\Rightarrow (7)$.
The implication $(7)\Rightarrow (8)$ is evident and $(8)\Rightarrow (2)$ follows from \cref{sym}.
\end{proof}

\begin{lem}\label{Q}
Let $K/F$ be a separable quadratic extension and let $\mathfrak{b}$ be a symmetric bilinear form over $F$.
Then $D(\mathfrak{b}_K)\cap F=D(\mathfrak{b})$.
\end{lem}
\begin{proof}
Clearly, we have $D(\mathfrak{b})\subseteq D(\mathfrak{b}_K)\cap F$.
Suppose that $\alpha\in D(\mathfrak{b}_K)\cap F$.
Let $V$  be an underlying vector space of $\mathfrak{b}$.
Write $K=F(\eta)$ for some $\eta\in K$ with $\delta:=\eta^2+\eta\in F$.
Then $\alpha=\mathfrak{b}_K(u\otimes1+v\otimes\eta,u\otimes1+v\otimes\eta)$ for some $u,v\in V$.
Thus
\begin{align*}
\alpha &=(\mathfrak{b}(u,u)+\delta\mathfrak{b}(v,v))+\eta(\mathfrak{b}(u,v)+\mathfrak{b}(v,u)+\mathfrak{b}(v,v))\\
&=(\mathfrak{b}(u,u)+\delta\mathfrak{b}(v,v))+\eta\mathfrak{b}(v,v).
\end{align*}
Since $\alpha\in F$ we get $\mathfrak{b}(v,v)=0$.
Hence, $\alpha=\mathfrak{b}(u,u)\in D(\mathfrak{b})$, proving the result.
\end{proof}

\begin{prop}\label{QP}
If $(A,\sigma)$ is a totally decomposable algebra with orthogonal involution over $F$, then $D(q_\sigma)=D(\mathfrak{Pf}(A,\sigma))$.
\end{prop}

\begin{proof}
Let $\mathfrak{b}=\mathfrak{Pf}(A,\sigma)$.
Since $\mathfrak{b}(x,x)=x^2=q_\sigma(x)$ for every $x\in\Phi(A,\sigma)\subseteq S(A,\sigma)$ we have $D(\mathfrak{b})\subseteq D(q_\sigma)$.
To prove the converse inclusion, let $(A,\sigma)\simeq\bigotimes_{i=1}^n(Q_i,\sigma_i)$ be a decomposition of $(A,\sigma)$ into quaternion algebras with involution.
For $i=0,\cdots,n$, define a field $K_i$ inductively as follows:
set $K_0=F$ and suppose that $K_i$ is defined.
If $K_i$ splits $A$, set $K_{i+1}=K_i$.
Otherwise, let $r$ be the minimal number for which $K_i$ does not split $Q_r$.
Then $Q_r\otimes_F K_i$ is a division algebra over $K_i$.
Let $K_{i+1}$ be a maximal separable subfield of $Q_r\otimes K_i$.
Note that for $i=0,\cdots,n-1$, $K_i$ may be identified with a subfield of $K_{i+1}$.
Also, either $K_{i+1}=K_i$ or $K_{i+1}/K_i$ is a separable quadratic extension.
Set $L:=K_n$, so that $A_L$ splits.

By \cite[(7.5)]{dolphin3}, we may identify $(A,\sigma)_L=(\End_L(V),\sigma_{\mathfrak{b}_L})$.
Using \cref{split} we get $D(q_{\sigma_L})\subseteq D(\mathfrak{b}_L)$.
If $x\in S(A,\sigma)$ and $\alpha=q_\sigma(x)$, then $x\otimes1\in S((A,\sigma)_L)$ and $q_{\sigma_L}(x\otimes1)=\alpha\otimes1$.
Hence, by identifying $F\otimes F\subseteq A\otimes F$ with $F$ we have $D(q_\sigma)\subseteq D(q_{\sigma_L})$.
It follows that $D(q_\sigma)\subseteq D(\mathfrak{b}_L)$.
By \Cref{Q} and induction on $n$ we have $D(\mathfrak{b}_L)\cap F=D(\mathfrak{b})\subseteq D(q_{\sigma_L})$.
Hence, $D(q_\sigma)\subseteq D(\mathfrak{b})$, proving the result.
\end{proof}

\begin{lem}\label{b}
Let $\mathfrak{b}$ and $\mathfrak{b}'$ be two isotropic bilinear $n$-fold Pfister forms over $F$.
Then $\mathfrak{b}\simeq\mathfrak{b}'$ if and only if $Q(\mathfrak{b})=Q(\mathfrak{b}')$.
\end{lem}

\begin{proof}
The `only if' implication is evident.
To prove the converse, choose positive integers $r$ and $r'$ and anisotropic bilinear Pfister forms $\mathfrak{c}$ and $\mathfrak{c}'$ over $F$
such that $\mathfrak{b}\simeq\lla1\rra^r\otimes\mathfrak{c}$ and $\mathfrak{b}'\simeq\lla1\rra^{r'}\otimes\mathfrak{c}'$, where $\lla1\rra^s$ is the $s$-fold Pfister form $\lla1,\cdots,1\rra$ (see \cite[p. 909]{arason}).
Since $\mathfrak{c}$ and $\mathfrak{c}'$ are anisotropic we have $\dim_{F^2}Q(\mathfrak{c})=2^{n-r}$ and $\dim_{F^2}Q(\mathfrak{c}')=2^{n-r'}$.
As $Q(\mathfrak{b})=Q(\mathfrak{c})$ and $Q(\mathfrak{b}')=Q(\mathfrak{c}')$, the assumption implies that $Q(\mathfrak{c})=Q(\mathfrak{c}')$, hence $r=r'$.
The conclusion now follows from \cite[(A.8)]{arason}.
\end{proof}

\begin{cor}\label{co}
Let $(A,\sigma)$ and $(A',\sigma')$ be totally decomposable algebras with isotropic orthogonal involution over $F$.
If $q_\sigma\simeq q_{\sigma'}$ then $\mathfrak{Pf}(A,\sigma)\simeq\mathfrak{Pf}(A',\sigma')$.
\end{cor}

\begin{proof}
The result follows from \cref{QP} and \cref{b}.
\end{proof}

\begin{thm}\label{main1}
Let $(A,\sigma)$ and $(A',\sigma')$ be totally decomposable algebras with isotropic orthogonal involution over $F$.
Then $(A,\sigma)\simeq(A',\sigma')$ if and only if $A\simeq A'$ and $q_\sigma\simeq q_{\sigma'}$.
\end{thm}

\begin{proof}
The `only if' implication follows from \cref{iso} and the converse follows from \cref{co} and \cite[(6.5)]{mn1}.
\end{proof}

\begin{notation}
Let $(A,\sigma)$ be a totally decomposable algebra with anisotropic orthogonal involution over $F$.
By \cref{isotropic} we have $S(A,\sigma)=\Phi(A,\sigma)\subseteq F\oplus \alt(A,\sigma)$.
We denote the set $S(A,\sigma)\cap\alt(A,\sigma)$ by $S'(A,\sigma)$.
We also denote by $q_\sigma'$ the restriction of $q_\sigma$ to $S'(A,\sigma)$.
\end{notation}
Note that we have $S(A,\sigma)=F\oplus S'(A,\sigma)$ and $q_\sigma\simeq\langle1\rangle_q\perp q'_\sigma$, because $q_\sigma(1)=1$.
Also, as observed in \cite[p. 223]{mn1} one has an orthogonal decomposition $\Phi(A,\sigma)=F\perp S'(A,\sigma)$ with respect to $\mathfrak{Pf}(A,\sigma)$.
It follows that $q'_\sigma(x)=x^2=\mathfrak{b}'(x,x)$ for $x\in S'(A,\sigma)$, where $\mathfrak{b}'$ is the pure subform of $\mathfrak{Pf}(A,\sigma)$.
Hence, we have the following result
(recall that for a symmetric bilinear space $(V,\mathfrak{b})$ over $F$, there exists a unique totally singular quadratic form $\phi_\mathfrak{b}$ on $V$ given by $\phi_\mathfrak{b}(x)=\mathfrak{b}(x,x)$).
\begin{lem}\label{phi}
Let $(A,\sigma)$ be a totally decomposable algebra with anisotropic orthogonal involution over $F$ and let $\mathfrak{b}=\mathfrak{Pf}(A,\sigma)$.
Then $q_\sigma=\phi_\mathfrak{b}$ and $q'_\sigma=\phi_{\mathfrak{b}'}$, where $\mathfrak{b}'$ is the pure subform of $\mathfrak{b}$.
\end{lem}

Let $(A,\sigma)$ be a totally decomposable algebra of degree $2^n$ with orthogonal involution over $F$.
Then there exists a set $\{v_1,\cdots,v_n\}$ consisting of units
such that $\Phi(A,\sigma)\simeq F[v_1,\cdots,v_n]$ and
$v_{i_1}\cdots v_{i_s}\in\alt(A,\sigma)$ for every $1\leqslant s\leqslant n$ and $1\leqslant i_1<\cdots<i_s\leqslant n$ (see \cite[(5.1)]{mn1} for more details).
As in \cite{mn1} we call $\{v_1,\cdots,v_n\}$ a {\it set of alternating generators} of $\Phi(A,\sigma)$.
According to \cite[(5.3) and (5.5)]{mn1}, if $\{v_1,\cdots,v_n\}$ is a set of alternating generators of $\Phi(A,\sigma)$ and $\alpha_i=v_i^2\in F^\times$ for $i=1,\cdots,n$, then $\mathfrak{Pf}(A,\sigma)\simeq\lla\alpha_1,\cdots,\alpha_n\rra$.
\begin{prop}\label{qsig}
Let $(A,\sigma)$ and $(A',\sigma')$ be totally decomposable algebras with anisotropic orthogonal involution over $F$.
Then $q'_\sigma\simeq q'_{\sigma'}$ if and only if $\mathfrak{Pf}(A,\sigma)\penalty 0\simeq\mathfrak{Pf}(A',\sigma')$.
\end{prop}

\begin{proof}
The `if' implication follows from \cref{phi}.
To prove the converse, let  $\deg_FA=2^n$ and let $\{x_1,\cdots,x_n\}$ be a set of alternating generators of $\Phi(A,\sigma)$.
Set $\alpha_i=x_i^2\in F^\times$, so that $\mathfrak{Pf}(A,\sigma)\simeq \lla\alpha_1,\cdots,\alpha_n\rra$.
By dimension count the set
\[\{x_{i_1}\cdots x_{i_s}\mid 1\leqslant s\leqslant n \ {\rm and} \ 1\leqslant i_1<\cdots<i_s\leqslant n\},\]
is a basis of $S'(A,\sigma)$ over $F$.
Let $f:(S'(A,\sigma),q'_\sigma)\iso(S'(A',\sigma'),q'_{\sigma'})$ be an isometry and set $x'_i=f(x_i)\in S'(A',\sigma')$, $i=1,\cdots,n$.
Then
\begin{equation}\label{eq6}
x_i^{\prime2}=q'_{\sigma'}(x'_i)=q'_\sigma(x_i)=\alpha_i\quad {\rm for} \ i=1,\cdots,n.
\end{equation}
We claim that $f(x_{i_1}\cdots x_{i_s})=x'_{i_1}\cdots x'_{i_s}$ for $1\leqslant s\leqslant n$  and $1\leqslant i_1<\cdots<i_s\leqslant n$.
Since $S(A',\sigma')$ is an $F$-algebra we have $x'_{i_1}\cdots x'_{i_s}\in S(A',\sigma')$.
We also have
\begin{align}\label{eq7}
q_{\sigma'}(f(x_{i_1}\cdots x_{i_s})+x'_{i_1}\cdots x'_{i_s})&=q_\sigma(x_{i_1}\cdots x_{i_s})+q_{\sigma'}(x'_{i_1}\cdots x'_{i_s})\nonumber\\
&=(x_{i_1}\cdots x_{i_s})^2+(x'_{i_1}\cdots x'_{i_s})^2.
\end{align}
Since $S(A,\sigma)=\Phi(A,\sigma)$, $S(A,\sigma)$ is commutative.
Similarly, $S(A',\sigma')$ is also commutative.
Hence, using (\ref{eq7}) and (\ref{eq6}) we get
\[q_{\sigma'}(f(x_{i_1}\cdots x_{i_s})+x'_{i_1}\cdots x'_{i_s})=x_{i_1}^2\cdots x_{i_s}^2+x^{\prime2}_{i_1}\cdots x^{\prime2}_{i_s}=0.\]
As $q_{\sigma'}$ is anisotropic we get $f(x_{i_1}\cdots x_{i_s})=x'_{i_1}\cdots x'_{i_s}$, proving the claim.
In particular, we have $x'_{i_1}\cdots x'_{i_s}\in S'(A',\sigma')\subseteq\alt(A',\sigma')$ for $1\leqslant s\leqslant n$  and $1\leqslant i_1<\cdots<i_s\leqslant n$.
Hence, $\{x'_1,\cdots,x'_n\}$ is a set of alternating generators of $\Phi(A',\sigma')$.
The relation (\ref{eq6}) now implies that $\mathfrak{Pf}(A',\sigma')\simeq\lla\alpha_1,\cdots,\alpha_n\rra\simeq\mathfrak{Pf}(A,\sigma)$.
\end{proof}

Using \cref{iso}, \cite[(6.5)]{mn1} and \cref{qsig} we have the following analogue of \cref{main1}.
\begin{thm}\label{main2}
Let $(A,\sigma)$ and $(A',\sigma')$ be totally decomposable algebras with anisotropic orthogonal involution over $F$.
Then $(A,\sigma)\simeq(A',\sigma')$ if and only if $A\simeq A'$ and $q'_\sigma\simeq q'_{\sigma'}$.
\end{thm}

\noindent{\sc A.-H. Nokhodkar, {\tt
    anokhodkar@yahoo.com},\\
Department of Pure Mathematics, Faculty of Science, University of Kashan, P.~O. Box 87317-51167, Kashan, Iran.}


\footnotesize
\begin{thebibliography}{MM}
\bibitem{arason} J. Arason, R. Baeza,
Relations in $I^n$ and $I^nW_q$ in characteristic $2$.
{\it J. Algebra} {\bf 314} (2007), no. 2, 895--911.

\bibitem{dolphin2} A. Dolphin, Decomposition of algebras with involution in characteristic $2$. {\it J. Pure Appl. Algebra} {\bf 217} (2013), no. 9, 1620--1633.

\bibitem{dolphin3} A. Dolphin, Orthogonal Pfister involutions in characteristic two. {\it J. Pure Appl. Algebra} {\bf 218} (2014), no. 10, 1900--1915.

\bibitem{elman}
R. Elman, N. Karpenko, A. Merkurjev,
{\it The algebraic and geometric theory of quadratic forms.}
American Mathematical Society Colloquium Publications, 56. American Mathematical Society, Providence, RI, 2008.

\bibitem{knus} M.-A. Knus, A. S. Merkurjev, M. Rost, J.-P. Tignol, {\it The book of involutions}. American Mathematical Society Colloquium Publications, 44. American Mathematical Society, Providence, RI, 1998.

\bibitem{lagh}
A. Laghribi, P. Mammone. Hyper-isotropy of bilinear forms in characteristic $2$. {\it Contemporary Mathematics}, {\bf 493} (2009), 249–-269.

\bibitem{mn1} M. G. Mahmoudi, A.-H. Nokhodkar, On totally decomposable algebras with involution in characteristic two. {\it J. Algebra} {\bf 451} (2016), 208--231.
\end{thebibliography}
\end{document}